\documentclass[11pt]{amsart}
\usepackage{amsmath}
\usepackage{amsthm}
\usepackage{amscd}
\usepackage{amssymb}
\usepackage{amsfonts}
\usepackage{color}
\usepackage[dvipsnames]{xcolor}
\usepackage{comment}
\usepackage{color}

\newlength{\defbaselineskip }
\long\def\salta#1{\relax}

\theoremstyle{plain}
\newtheorem{theorem}{Theorem}[section]
\newtheorem{proposition}[theorem]{Proposition}
\newtheorem{lemma}[theorem]{Lemma}

\theoremstyle{definition}
\newtheorem{remark}[theorem]{Remark}

\newcommand{\re}{\mathbb{R}}

\def\into{\displaystyle \int_{\Omega}}
\def\sobr{W^{1,r}_{0}(\Omega)}

\numberwithin{equation}{section}

\title[Uniform $L^{\infty}$-boundedness for anisotropic quasilinear systems]{Uniform $L^{\infty}$-boundedness for solutions of anisotropic quasilinear systems}

\author[N. Borgia]{Natalino Borgia}
\address{Dipartimento di  Matematica  \\ Universit\`{a} degli Studi di Bari Aldo Moro \\ Via Orabona 4\\ 70125 Bari, Italy}
\email{natalino.borgia@uniba.it}
\author[S. Cingolani]{Silvia Cingolani}
\address{Dipartimento di  Matematica  \\ Universit\`{a} degli Studi di Bari Aldo Moro \\ Via Orabona 4\\ 70125 Bari, Italy}
\email{silvia.cingolani@uniba.it}
\author[G. Vannella]{Giuseppina Vannella}
\address{Dipartimento di Meccanica, Matematica e Management \\ Politecnico di Bari\\ Via Orabona 4\\ 70125 Bari, Italy}
\email{giuseppina.vannella@poliba.it}

\begin{document}

\begin{abstract}
In this paper we obtain uniformly locally $L^{\infty}$-estimate of solutions to non-autonomous quasilinear system involving operators in divergence form and a family of nonlinearities that are allowed to grow also critically.
\end{abstract}



\maketitle

\section{Introduction}

\noindent
In this paper we are interested to study $L^{\infty}$-regularity of solutions of the following system

\begin{equation}\label{pq}
	\begin{cases}
		\begin{array}{ll}
			-\text{\rm div}  \; [ \nabla \Psi_1 \left( \nabla u  \right)] = H_s(\delta,x, u,v) & \hbox{in} \ \Omega,
			\medskip \\
			-\text{\rm div} \; [ \nabla \Psi_2 \left( \nabla v  \right)]= H_t(\delta,x,u,v) & \hbox{in} \ \Omega, \medskip\\
			u=v=0  & \hbox{on} \ \partial\Omega,
		\end{array}
	\end{cases}
\end{equation}

\noindent
where $\Omega$ is a smooth bounded domain of $\mathbb{R}^N$,  $ N \geq 2 $.\\
Denoting with $I$ an interval of $\mathbb{R}$, the  function $H:I\times \Omega \times \re^2 \to \mathbb{R}$  is such that $H(\delta,x,\cdot,\cdot) \in C^1(\re^2, \re)$ for any $\delta \in I$ and for a.e. $x\in \Omega$, $H_s(\delta, \cdot,s,t)$, $H_t(\delta, \cdot,s,t)$ are measurable for every $(\delta,s,t) \in I\times \re^2$, and $H(\delta,x,0,0)$ belongs to $ L^{1}(\Omega)$ for any $\delta \in I$.\\
Moreover, we assume:
\begin{itemize}
\item[$(\Psi)$] the functions $\Psi_{1}, \Psi_{2}: \mathbb{R}^N \to \mathbb{R}$ are of class $C^1$ with $\Psi_1(0)=0$, $\nabla \Psi_1(0)=0$  and $\Psi_2(0)=0$, $\nabla \Psi_2(0)=0$. Moreover, given $k \geq 0,$  $r >1$ and denoting with $\Psi_{r,k}: \mathbb{R}^N \to \mathbb{R}$ the function defined as
\begin{equation}\label{rArea}
	\displaystyle \Psi_{r,k}(\xi):= \frac{1}{r} \biggl[ \left( k^2 +
	\left| \xi \right|^2\right)^{ \frac{r}{2}} - k^r \biggr] ,
\end{equation}	
we assume that there exist $\alpha \geq 0$, $1<p<N$ and $ \frac{1}{p} < \nu_1 \leq C_1$ such that $ \displaystyle \left( \Psi_1 - \nu_1 \Psi_{p,\alpha}    \right)$  and  $ \displaystyle \left( C_1 \Psi_{p,\alpha} - \Psi_1   \right)$  are both convex, and there exist $\beta \geq 0$, $1<q<N$ and $ \frac{1}{q} < \nu_2 \leq C_2$ such that $ \displaystyle \left( \Psi_2 - \nu_2 \Psi_{q,\beta}    \right)$ and $ \displaystyle \left( C_2 \Psi_{q,\beta}   - \Psi_2 \right)$ are both convex.\\
	
\item[$(\mathcal{H})$]
	there exists $C_0>0$ such that
	\[|H_s(\delta,x,s,t)|\leq C_0 \left( 1 + |s|^{p^*-1} +|t|^{q^*\frac{p^*-1}{p^*}}   \right)
	\]
	\[|H_t(\delta,x,s,t)|\leq C_0 \left( 1 + |s|^{p^*\frac{q^*-1}{q^*}}+ |t|^{q^*-1}  \right),
	\]
	for a.e. $x\in \Omega$ and every $(\delta,s,t)\in I\times \re^2$, where $p^*:=Np/(N-p)$ and $q^*:=Nq/(N-q)$ are the critical Sobolev exponents of $p$ and $q$ respectively, where $p$ and $q$ are introduced in assumption $(\Psi)$.

\end{itemize}

Systems involving this kind of quasilinear operators 
model some phenomena in non-Newtonian mechanics, nonlinear elasticity and glaciology, combustion theory, population biology (see \cite{boccardodefiguerido,diazthelin, glow, manamaw,marcellini1}).

\medskip
\noindent
Let $X$ be the product space $W_0^{1,p}(\Omega)\times
W_0^{1,q}(\Omega)$  endowed with the norm
\[
\|z\|= \|u\|_{1,p} + \|v\|_{1,q},
\]
where $z=(u,v)\in X$.  Throughout this paper we shall denote respectively
by $\| \cdot\|_r$ and $\| \cdot \|_{1,r}$ the usual  norms in
$L^r(\Omega)$ and $W^{1,r}_0(\Omega)$.

Weak solutions of problem \eqref{pq} correspond to critical points
of the Euler functional $I_{\delta, \Psi_1, \Psi_2}:X\to \mathbb{R}$ defined as
\begin{align*}
	I_{\delta,\Psi_1,\Psi_2}(z) =I_{\delta, \Psi_1, \Psi_2}(u,v) =&  \into  \Psi_1(\nabla u) dx + \into \Psi_2(\nabla v ) dx  \\ &
	- \into
	H(\delta,x,u(x),v(x)) \,dx ,  
\end{align*}
for any $z=(u,v) \in X$.\\
Moreover, by $(\mathcal{H})$ and $(\Psi)$,  the functional $I_{\delta,\Psi_1, \Psi_2}$ is $C^1$ on $X$.

Condition $(\Psi)$ was introduced in \cite{CDV} where the authors obtained regularity results for the scalar case of system \eqref{pq}, with a nonlinearity independent of the variable $\delta$ that is allowed to grow critically. (See also \cite{dibenedetto1983, GV, lieberman1988,tolksdorf1983, tolksdorf1984} and \cite{CY,HO} for recent results). We refer the reader also to \cite{ACCFM,ACF} in which, letting $ B(t):= t^p/p \text{ for } t>0,$ the authors considered 
$$ \displaystyle  \Psi_1=B \circ H,$$
where the norm $H: \mathbb{R}^N \to \mathbb{R}$ is of class $C^2$ in $\mathbb{R}^N \setminus \{O \} $, and  its anisotropic unit ball is uniformly convex (see for instance \cite{CFV}).

Quasilinear elliptic systems were studied in \cite{MW}, where the authors obtained $L^{\infty}$-estimates results under general conditions on the data (see assumption $(\tilde{H})$ and Theorem 3.4. in \cite{MW}) that allow them to apply Moser’s iteration technique (see also \cite{CSS,DeNM,GPZ,MP} for existence results and \cite{BDMS,CM,DeFP,MINGIONE,MMSV} for regularity results).
Our aim is to prove $L^\infty$-estimates for system of the form \eqref{pq}, under assumptions $(\Psi)$ and $(\mathcal{H})$ that it turns out are quite general.

It is immediate to see that, for any $1<p<N,$ $ \alpha \geq 0$ and $1<q<N,$  $ \beta \geq 0$, the functions $\Psi_1 =\Psi_{p,\alpha}$ and $\Psi_2 =\Psi_{q,\beta}$ satisfy assumption $(\Psi)$  respectively (it is enough to take $\nu_1=C_1=1$ and $\nu_2=C_2=1$ respectively). Therefore, taking into account that 

\begin{equation}\label{rAreaGradiente}
\displaystyle  
\nabla \Psi_{r,k}(\xi):=
\begin{cases}
\left( k^2 +
	\left| \xi \right|^2\right)^{ \frac{r-2}{2}} \xi  & \text{ if } \xi \neq 0, \bigskip \\
0 & \text{ if } \xi=0,
\end{cases}
\end{equation}
the study of the system \eqref{pq} includes the study of system involving $p$-Laplacian and $q$-Laplacian operators or $p$-area type and $q$-area type operators.

In fact, if  we take $\Psi_1 =\Psi_{p,0}$ and $\Psi_2 =\Psi_{q,0}$, and moreover we take a function $H$ that does not depend on the variable $\delta$, system \eqref{pq} becomes

\begin{equation}\label{pqBoccardo}
	\begin{cases}
		\begin{array}{ll}
			- \Delta_p u = H_s(x, u,v) & \hbox{in} \ \Omega, \medskip
			\\
			- \Delta_q v = H_t(x,u,v) & \hbox{in} \ \Omega, \medskip\\
			u=v=0  & \hbox{on} \ \partial\Omega.
		\end{array}
	\end{cases}
\end{equation}

When $p=q=2$, a priori bounds for semilinear elliptic systems of the form \eqref{pqBoccardo}, or similar, can be found in \cite{CDM1,CDM2,DeY,QS}       with different methods and techniques.

Now, if we  consider $\Psi_1 =\Psi_{p,\alpha}$ and $\Psi_2 =\Psi_{q,\beta}$ with $\alpha, \beta \geq 0$, by \eqref{rAreaGradiente} problem \eqref{pq} becomes

\begin{equation}\label{pqcasospecifico}
	\begin{cases}
		\begin{array}{ll}
			-\text{\rm div} \left( (\alpha^2+|\nabla u|^2)^{\frac{p-2}{2}}\nabla u\right) = H_s(\delta,x, u,v) & \hbox{in} \ \Omega;
			 \medskip \\
			-\text{\rm div} \left( (\beta^2+|\nabla v|^{2})^{\frac{q-2}{2}}\nabla v\right)= H_t(\delta,x,u,v) & \hbox{in} \ \Omega; \medskip \\
			u=v=0  & \hbox{on} \ \partial\Omega.
		\end{array}
	\end{cases}
\end{equation}

Estimate of the $L^{\infty}$-norm of weak solutions of  \eqref{pqcasospecifico} is not straightforward, due to the coupling of $u$ with $v$ in $(\mathcal{H})$.
Results were obtained in \cite{CCMV}, under the conditions  $\alpha=\beta,$ $2\leq p,q < N$ and $H: \mathbb{R}^2 \to \mathbb{R}$ subcritical. 

Recently, in \cite{v}, the third author extended the estimates obtained in \cite{CCMV} to systems where the subcritical nonlinearity $H$ depends also on $\delta$,  so that we could use these new estimates in \cite{BCV}. 

In the present paper, taking into account assumptions $(\Psi)$ and $(\mathcal{H})$ we extend results obtained for system \eqref{pqcasospecifico}  in multiple directions. We consider $1<p,q<N$, two possibly different constants $\alpha$ and $\beta$, a nonlinearity $H$ that is allowed to grow critically, and moreover can explicitly depend on the independent variable $x$  and a parameter $\delta$, and finally we consider more general operators involving the functions $\Psi_1$ and $\Psi_2$ satisfying condition $(\Psi)$.

As already underlined, the extension to systems is not straightforward due to the coupling of $u$ with $v$ in $(\mathcal{H})$, and to overcome this difficulty we follow the same technique used in \cite{CCMV} and \cite{v}.

To extend also to the singular case, that is when $1 < p < 2$ or $1 < q < 2$, we cannot apply classical Stampacchia Lemma (see \cite[Lemma 4.1]{S}) but we can apply \cite[Lemma 2.1]{BDO} (see Section 2 for the statement) whose proof is based on the following generalization (see \cite[Lemma A.1]{BDO}) of Stampacchia Lemma

\begin{lemma} Let $\varphi: \mathbb{R}^+ \to \mathbb{R}^+$ be a non increasing function such that 

$$ \displaystyle \varphi(h) \leq \frac{c_0}{(h-k)^{\rho}}k^{\theta \rho} \, [\varphi(k)]^{1+ \lambda} \qquad \forall \, h>k \geq k_0,$$

\medskip
\noindent
where $c_0>0$, $k_0 \geq 0$, $\rho > 0$, $0 \leq \theta < 1$ and $\lambda>0$. Then there exists $k^*>0$ such that $\varphi(k^*)=0$.
\end{lemma}

Finally, since we are taking into account a nonlinearity $H$ that is allowed to grow critically, we will generalize the instrumental lemma \cite[Lemma 2.1]{v} (see Section 2 for the statement), provided we take a suitably small radius $r$:

\begin{lemma}\label{V0}
Let $s\in (1,N)$ and denote by $s^{*}$  the conjugate Sobolev exponent of~$s$.
For any $\varepsilon>0$ and $u_0\in W^{1,s}_0(\Omega)$,  there exist $\sigma>0$ and $r>0$ such that

\[\int_{\{|u(x)|\,\geq \sigma\}} \hspace{-1mm}|u(x)|^{s^*}\,  dx<\varepsilon \]
	
\medskip 
\noindent
for any $u\in B_r(u_0)=\{u\in  W^{1,s}_0(\Omega)\ :\ \|u-u_0\|_{1,s}\leq r \}$.
\end{lemma}

\bigskip
\noindent
We can now state our main result: 
\begin{theorem}\label{key} 
If $(u,v)$ is a solution of $(\ref{pq})$ and $(\mathcal{H})$ and $(\Psi)$ hold, then $(u,v)\in \left(L^\infty(\Omega)\right)^2$.
 Moreover, for any fixed $(u_0,v_0)\in X$  there exists a  suitably small $R >0$ such that for any $\delta \in I$ and  denoting by 

	\[ 
	D_{R, \delta}(u_0,v_0)=\{
	(u,v)\in 
	X\ :\  \| (u,v)-(u_0,v_0)\|\leq R, \   I'_{\delta,\Psi_1,\Psi_2}(u,v)=0
	\},
	\]

\medskip
\noindent
we have

	\[ \|u\|_\infty,\, \|v\|_\infty\, \leq C
	\qquad  \forall \, (u,v)\in D_{R, \delta}(u_0,v_0),
	\]

\medskip	
\noindent
 where both the positive constants $R$ and $C$ depend on $u_0, v_0, C_0,\left| \Omega \right|, p,q, N,$ $ \alpha, \beta, \nu_1$ and $\nu_2$, but not on $\delta$.
		
\end{theorem}

So in the present paper we show carefully that, for any arbitrary $z_0 \in X$  there exists $R>0$  such that  the $\left(L^\infty(\Omega)\right)^2$-norm of any weak solution to (\ref{pq}) belonging to $B_R(z_0)$ is bounded by a constant $C$ that is independent of $\delta \in I $.

As a consequence of Theorem \ref{key}, we obtain $L^{\infty}$-boundedness of solutions also for system \eqref{pqBoccardo}. System of the form \eqref{pqBoccardo} include also different types of eigenvalue problem, such as

\begin{equation}\label{eigenZOGR}
	\begin{cases}
		\begin{array}{ll}
			- \Delta_p u  = \lambda |u|^{p-2} u +
			\frac{\lambda}{\gamma +1}|u|^{\beta} |v|^{\gamma} v, & \text{ in } \Omega,
			\medskip \\
			- \Delta_q v = \lambda |v|^{q-2} v +
			\frac{\lambda}{\beta +1}|u|^{\beta} |v|^{\gamma} u, & \text{ in } \Omega, \medskip \\
			u=v=0,  & \text{ on  }  \partial\Omega,
		\end{array}
	\end{cases}
\end{equation}

where $1 <p,q < N$ and $\beta, \gamma \geq 0$ satisfy $(\beta +1)/ p + (\gamma +1)/q =1$.\\
By our main result we deduce that if $(u,v)$ is an eigenfunction associated to the  eigenvalue $\bar{\lambda}$ of \eqref{eigenZOGR}, then both $ u \in L^\infty(\Omega)$ and $v \in L^{ \infty}(\Omega)$, and their norms in $L^\infty(\Omega)$ are bounded by a constant that depends on $\bar{\lambda}$.

\section{PRELIMINARY RESULTS}

\medskip
\noindent
The Euler functional associated to system \eqref{pq} is $I_{ \delta, \Psi_1, \Psi_2}: X \to \mathbb{R}$ defined as

\begin{align}\label{funzionale}
	I_{\delta,\Psi_1,\Psi_2}(z) =I_{ \delta, \Psi_1, \Psi_2}(u,v) =&  \into  \Psi_1(\nabla u) dx + \into \Psi_2(\nabla v ) dx  \\ &
	- \into
	H(\delta,x,u(x),v(x)) \,dx,   \nonumber
\end{align}
for any $z=(u,v) \in X$.

By assumption  $(\mathcal{H})$, we have that the part of \eqref{funzionale} involving the nonlinearity $H$ is of class $C^1$ on $X$.

To prove that assumption $(\Psi)$ implies that also the part of \eqref{funzionale} involving the anisotropic-type operators is of class $C^1$ on $X$, we need the following result.

\begin{lemma}\label{contogradconvess} Let  $\Psi: \mathbb{R}^N \to \mathbb{R}$ a convex $C^1$ function. Assume that there exists $c>0, d>0$ and $r>1$ such that

\begin{equation}\label{proprconvPsi}
\displaystyle \left| \Psi(\xi) \right| \leq c\left|\xi \right|^r + d  \qquad \forall \, \xi \in \mathbb{R}^N.
\end{equation}

\medskip
\noindent
Then there exists  a constant $\tilde{c}$ such that

\begin{equation}\label{proprconvgradPsi}
\displaystyle \left| \nabla \Psi(\xi) \right| \leq \tilde{c} \left( \left|\xi \right|^{r-1}  + 1 \right) \qquad \forall \, \xi \in \mathbb{R}^N.
\end{equation}
\end{lemma}

\begin{proof} By convexity of $\Psi$, we know that

\begin{equation}\label{contoetah}
 \displaystyle \Psi(\xi + h) - \Psi(\xi) \geq \nabla \Psi(\xi) \cdot h \qquad \forall \, \xi, h \in \mathbb{R}^N.
\end{equation}

\medskip
\noindent
Let $\xi \in \mathbb{R}^N$ with $\xi \neq 0$. Taking into account \eqref{proprconvPsi} and \eqref{contoetah}, there exist $\bar{c}$ and $\tilde{d}$ such that
\begin{align*}
\displaystyle 
\left| \nabla \Psi(\xi) \right| = \sup_{ |v| \leq 1 } \; \nabla \Psi (\xi) \cdot v  = \sup_{ |v| \leq 1 } \; \frac{\nabla \Psi (\xi) \cdot \left|\xi \right| \, v}{\left| \xi \right|}
&  \leq \sup_{ |v| \leq 1 } \frac{\Psi(\xi + \left|\xi\right|  v) - \Psi(\xi)}{\left| \xi \right|} \\
&  \leq \sup_{ |v| \leq 1 } \frac{\left| \Psi(\xi + \left|\xi\right|  v) - \Psi(\xi) \right|}{\left| \xi \right|} \\
& \leq \bar{c} \left| \xi \right|^{r-1} + \frac{\tilde{d}}{\left| \xi \right|},
\end{align*}
\medskip
\noindent
by which, taking moreover into account that $\Psi$ is a function of class $C^1$ on $\mathbb{R}^N$, we deduce that there exists $\tilde{c}>0$ such that 
\begin{equation*}
\displaystyle \left| \nabla \Psi(\xi) \right| \leq \tilde{c} \left( \left|\xi \right|^{r-1}  + 1 \right) \qquad \forall \, \xi \in \mathbb{R}^N.
\end{equation*}
\end{proof}

\noindent
By Lemma  \ref{contogradconvess} e and by applying standard arguments, if $\Omega$ is a bounded domain of $\mathbb{R}^N$ and  if $\Psi: \mathbb{R}^N \to \mathbb{R}$ is a convex $C^1$ function satisfying \eqref{proprconvPsi},  then the functional $f: W_0^{1,r}(\Omega) \to \mathbb{R}$ defined as 

\begin{equation*}
 \displaystyle f(u):= \into \Psi \left( \nabla u(x) \right ) \, dx
\end{equation*}

is of class $C^1$ on $W_0^{1,r}(\Omega)$ and

\begin{equation*}
 \displaystyle \langle f'(u_0), u \rangle =   \into \nabla \Psi \left( \nabla u_0(x) \right) \cdot \nabla u(x) \, dx 
\end{equation*}

for any $u_0, u \in W_0^{1,r}(\Omega)$.

\medskip
\noindent
Let us now consider a function $\Psi_1: \mathbb{R}^N \to \mathbb{R}$ satisfying assumption $(\Psi)$, that is  $\Psi_{1}$ is of class $C^1$ with $\Psi_1(0)=0$, $\nabla \Psi_1(0)=0$,  and there exist $\alpha \geq 0$, $p>1$ and $ \frac{1}{p} < \nu_1 \leq C_1$ such that $ \displaystyle \left( \Psi_1 - \nu_1 \Psi_{p,\alpha}    \right)$  and  $ \displaystyle \left( C_1 \Psi_{p,\alpha} - \Psi_1   \right)$  are both convex.
Under this assumptions, it follows that $\Psi_1$ is strictly convex and moreover  

\begin{equation}\label{disuguaglianza1}
\displaystyle \nu_1 \Psi_{p,\alpha}( \xi) \leq \Psi_1(\xi) \leq C_1 \Psi_{p,\alpha}( \xi) \qquad \forall \, \xi \in \mathbb{R}^N,
\end{equation}

and

\begin{equation}\label{contoconvessità}
\displaystyle \nu_1( \alpha^2 + \left| \xi \right|^2   )^{\frac{p-2}{2}} |\xi|^2 \leq  \nabla \Psi_1 ( \xi ) \cdot \xi  \leq  C_1( \alpha^2 + \left| \xi \right|^2   )^{\frac{p-2}{2}} |\xi|^2 \qquad \forall \, \xi \in \mathbb{R}^N.
\end{equation}

\medskip
\noindent
First of all we observe that by \eqref{contoconvessità}, it follows that if $\Psi_1$ satisfies assumption $(\Psi)$ for some $1<p<N$, then this $p$ is unique.
Moreover, by \eqref{disuguaglianza1}, $\Psi_1$ satisfies  \eqref{proprconvPsi}.  Similar results hold for a function $\Psi_2:\mathbb{R}^N \to \mathbb{R}$ satisfying assumption $(\Psi)$ with suitable constant $\beta \geq 0$, $q>1$ and $ \frac{1}{q} < \nu_2 \leq C_2$.
Therefore, by $(\mathcal{H})$ and $(\Psi)$  the functional $I_{\delta, \Psi_1, \Psi_2 }$ is of class $C^1$ on $X$ and
\begin{align*}
	\langle I'_{\delta,\Psi_1,\Psi_2}(z_0), z \rangle = & \displaystyle  \into \nabla \Psi_1 \left( \nabla u_0 \right) \cdot \nabla u \, dx +\into
	\nabla \Psi_2 \left( \nabla v_0  \right) \cdot \nabla v \,dx \\
	&-\displaystyle\into   \bigl(H_s(\delta,x,u_0,v_0) u +  H_t(\delta,x,u_0,v_0) v\bigr)\, dx
\end{align*}

for any $z_0=(u_0,v_0)$, $z=(u,v) \in X$.

\bigskip
\noindent
Let us recall the following result (see \cite[Lemma 2.1]{v}).

\medskip

\begin{lemma}\label{V2023}
	Let $s\in (1,N)$ and denote by $s^{*}$ 
	the conjugate Sobolev exponent of~$s$, namely ${s^*=sN/(N-s)}$.
	If $r,\varepsilon>0$, $u_0\in W^{1,s}_0(\Omega)$ and $ \tilde{s} \in [1,s^*)$, there is $\sigma>0$ such that
	\[\int_{\{|u(x)|\,\geq \sigma\}} \hspace{-1mm}
	|u(x)|^{\tilde{s}}\,  dx<\varepsilon
	\]
\medskip 
for any $u\in B_r(u_0)=\{u\in  W^{1,s}_0(\Omega)\ :\ \|u-u_0\|_{1,s}\leq r \}$.

\end{lemma}

Since in $(\mathcal{H})$ we are assuming that the function $H$ could grow critically, we need to improve the previous lemma up to critical growth. We can do this, provided we take a suitably small radius $r$.

\medskip
\noindent

\begin{lemma}\label{V}
Let $s\in (1,N)$ and denote by $s^{*}$  the conjugate Sobolev exponent of~$s$.
For any $\varepsilon>0$ and $u_0\in W^{1,s}_0(\Omega)$,  there exist $\sigma>0$ and $r>0$ such that

\[\int_{\{|u(x)|\,\geq \sigma\}} \hspace{-1mm}|u(x)|^{s^*}\,  dx<\varepsilon \]
	
\medskip 
\noindent
for any $u\in B_r(u_0)=\{u\in  W^{1,s}_0(\Omega)\ :\ \|u-u_0\|_{1,s}\leq r \}$.
\end{lemma}

\begin{proof}
By contradiction, assume that there are $\varepsilon>0$ and $u_0\in W^{1,s}_0(\Omega)$ such that, for any $\sigma>0$ and $r>0$ there exists $u_{\sigma,r} \in B_r(u_0)$ such that

$$
\displaystyle \int_{\{|u_{\sigma,r}(x)|\,\geq \sigma\}} \hspace{-1mm} |u_{\sigma,r}(x)|^{s^*}\,  dx \geq \varepsilon.
$$

\medskip
\noindent
Hence, we deduce that for any $r>0$ there exists a sequence $\{ u_n \}_n \subset B_r(u_0)$ such that

\begin{equation}\label{calcolo}
\displaystyle \qquad \qquad \quad \int_{\{|u_{n}(x)|\,\geq n\}} \hspace{-1mm} |u_{n}(x)|^{s^*}\,  dx \geq \varepsilon \qquad \forall \, n \in \mathbb{N}.
\end{equation}

\medskip
\noindent
Up to subsequence, there exists $\bar{u} \in W_0^{1,s}(\Omega)$ such that $\{u_n\}_n$ converges to $\bar{u}$ weakly in $W_0^{1,s}(\Omega)$ and strongly in $L^t(\Omega)$ for any $1 \leq t < s^*$.\\
Denoting with $E_n:= \{ x \in \Omega \, : \, \left| u_{n}(x) \right |\,\geq n \}$, by Chebyshev inequality we know that

$$ \displaystyle \left| E_n \right| \leq \frac{1}{n} \into |u_n| \, dx, $$

\noindent
by which we deduce that $\left| E_n \right| \to 0$ as $n \to + \infty$.\\

\medskip
\noindent
Observe now that

\begin{align}\label{calcolo2}
\displaystyle 
\int_{\{|u_n(x)|\,\geq n \}} \hspace{-1mm}|u_n(x)|^{s^*}\,  dx \nonumber
& = \int_{E_n} \hspace{-1mm}|u_n(x)- u_0(x) + u_0(x)|^{s^*}\,  dx \\
& \leq 2^{s^*} \int_{E_n} |u_n(x)- u_0(x)|^{s^*}\,  dx + 2^{s^*} \int_{E_n} |u_0(x)|^{s^*}\,  dx \nonumber\\
& \leq 2^{s^*} \into |u_n(x)- u_0(x)|^{s^*}\,  dx + 2^{s^*} \int_{E_n} |u_0(x)|^{s^*}\,  dx.
\end{align}

\medskip
\noindent
Since $\left| E_n \right| \to 0$ as $n \to + \infty$ and $u_0 \in L^{s^*}(\Omega)$, by absolute continuity of the Lebesgue integral we deduce 

$$ \displaystyle \int_{E_n} |u_0(x)|^{s^*}\,  dx \to 0 \qquad \text{as } n \to + \infty.$$
\medskip
\noindent
Taking into account Sobolev-Gagliardo-Nirenberg inequality, there exists a constant $C:=C(s,N)$ such that

$$ \displaystyle \into |u_n(x)- u_0(x)|^{s^*}\,  dx \leq C \|u_n-u_0\|_{1,s}^{s^*} \leq C r^{s^*} \qquad \forall n.$$

\medskip
\noindent
By choosing  $r= \frac{1}{2C^{\frac{1}{s^*}}} \bar{\varepsilon}^{\frac{1}{s^*}}$ with $0 < \bar{\varepsilon} < \varepsilon$ and passing to $ \displaystyle \limsup_{n \to + \infty}$ in \eqref{calcolo2}, we obtain 

$$ \displaystyle \limsup_{n \to + \infty} \int_{\{|u_n(x)|\,\geq n \}} \hspace{-1mm}|u_n(x)|^{s^*}\,  dx \leq \bar{\varepsilon},$$

\medskip
\noindent
that is a contradiction with \eqref{calcolo}. 
\end{proof}

\medskip
\noindent
To prove Theorem \ref{key}, we need to recall also the following result (see \cite[Lemma 2.1]{BDO}).

\begin{theorem}\label{Mammoliti}
Let $u$ a function in $ W_0^{1,p}(\Omega)$ such that, for $k$ greater then some $k_0$, 

\begin{equation}\label{inequalityBDO}
\displaystyle
\int_{A_k} \left| \nabla u  \right|^p \, dx \leq c k^{\theta p} \left|  A_k \right|^{\frac{p}{p^*} + \varepsilon},
\end{equation}

\medskip
\noindent
where $\varepsilon >0$, $0 \leq \theta < 1$, $p^*=\frac{Np}{N-p}$ and 
$$A_k= \left\{ x\in \Omega\ :\ \left|u (x) \right| > k\right\}.$$

\medskip
\noindent
Then, the norm of $u$ in $L^{\infty}(\Omega)$ is bounded by a constant that depends  on  $c,\theta,p, N, \varepsilon, k_0$ and $ \left| \Omega \right|$. 
\end{theorem}

\section{UNIFORMLY LOCALLY $L^{\infty}$-ESTIMATE}

\medskip
\noindent
Now, inspired by \cite{CCMV}, \cite{GV} and \cite{v},  we prove the main result. To this aim we need the following:

\begin{proposition}\label{integrability}
 For any fixed $(u_0,v_0)\in X$  there exists a  suitably small $R >0$ such that for any $\delta \in I$,  denoting by

	\[ 
	D_{R, \delta}(u_0,v_0)=\{
	(u,v)\in 
	X\ :\  \| (u,v)-(u_0,v_0)\|\leq R, \   I'_{\delta,\Psi_1,\Psi_2}(u,v)=0
	\},
	\]

\medskip
\noindent
and for any $\gamma>1$ we have

\begin{equation*}
	\into |\bar u|^{\gamma  p^*} \, dx \leq C \quad \text{and} \quad \into |\bar v|^{\gamma  q^*}\, dx  \leq  C  
	\qquad \qquad \forall (\bar u, \bar v) \in D_{R, \delta}(u_0,v_0),
\end{equation*}

\medskip	
\noindent
 where   both the positive constants $R$ and $C$ depend on $u_0, v_0, C_0,\left| \Omega \right|, p,q, N,$ $ \alpha, \beta, \nu_1$, $\nu_2$ and $\gamma$ but not on $\delta$.

\end{proposition}

\begin{proof} For simplicity, we will omit $dx$.
For every $\gamma,\, t,\, k>1$ we define
\begin{align*}
	h_{k,\gamma}(s)=&
	\begin{cases}
		s|s|^{\gamma-1} & \hbox{if }|s|\leq k,
		\\
		\gamma k^{\gamma -1} s + \text{sign}(s)(1-\gamma) k^{\gamma} & \hbox{if } |s|>k,
	\end{cases}
	\\
	\\
	\Phi_{k,t,\gamma}(s) =&\int_{0}^{s}\left| h'_{k,\gamma}(r)\right|^{\frac
		t\gamma}dr.
\end{align*}

Observe that $h_{k,\gamma}$ and $\Phi_{k,t,\gamma}$ are
$C^{1}$-functions with bounded derivative, depending on $\gamma, t$ and $k$. In particular, if
$w\in \sobr$, 
then $h_{k,\gamma}(w)$ and $\Phi_{k,t,\gamma}(w)$ are in $W_{0}^{1,r}(\Omega)$.

\smallskip

We see that 
\begin{equation}\label{hnew}
	|h'_{k,\gamma}(s)| \leq \gamma|h_{k,\gamma}(s)|^{\frac{\gamma-1}{\gamma}},
\end{equation}

and for every $t\geq \gamma$  there exists a positive constant $C$, depending on $\gamma$ and $t$ but independent of $k$, such that

\begin{equation}  \label{desgv}
	|s|^{\frac{t}{\gamma}-1}|\Phi_{k,t,\gamma}(s)|\leq C
	|h_{k,\gamma}(s)|^{\frac t\gamma}
\end{equation}
\begin{equation}\label{desgv2}
	|\Phi_{k,t,\gamma}(s)|\leq
	C |h_{k,\gamma}(s)|^{\frac1\gamma(1+t\frac {\gamma-1}\gamma)}
\end{equation}
and
\begin{equation}\label{desgv3}
	\left|h_{k,\gamma}\left(|s|^{\frac{q^*}{p^*}}\right)\right|^{p^*}
	\leq C\left|h_{k^{\frac{p^*}{q^*}},\gamma}\left(s\right)\right|^{q^*}.
\end{equation}

\medskip
\noindent
Moreover, for any $u \in  W_0^{1,p}(\Omega)$ we have

\begin{align}\label{gradhkgamma}
\displaystyle
 \nabla h_{k,\gamma}(u) = h'_{k,\gamma} (u) \nabla u,
\end{align}

and

\begin{align}\label{gradPhikgamma}
\displaystyle
\nabla \Phi_{k,\gamma p,\gamma }(u) = \Phi'_{k,\gamma p,\gamma } (u) \nabla \bar u = \left|h'_{k,\gamma} (u)\right|^p  \nabla u. 
\end{align}

\medskip
\noindent
For any $p>1$, $\alpha \geq 0$ and $s \geq 0$, we have

\begin{equation}\label{disr}
	s^p\leq \left(\alpha^2 +s^2\right)^{\frac{p-2}{2}}s^2  
	+\alpha^{p}.
\end{equation}

\medskip
\noindent
The inequality is obvious if $p\geq 2$ or $s=0$.\\
Otherwise, if $p\in(1,2)$ and $s \neq 0$, we have

$$ \displaystyle s^p \leq \left(\alpha^2 +s^2\right)^{\frac{p}{2}}=\left(\alpha^2 +s^2\right) \left(\alpha^2 +s^2\right)^{\frac{p-2}{2}}
\leq
\left(\alpha^2 +s^2\right)^{\frac{p-2}{2}}s^2 + \alpha^p.
$$

\medskip
\noindent
Let us fix $R>0$. For any  $\delta \in I$, let us consider  $\bar z=(\bar u, \bar v) \in D_{R,  \delta}(u_0,v_0)$.

In particular,  

$$ \left\langle I'_{\delta,\Psi_1,\Psi_2}(\bar z), \left( \Phi_{k,\gamma p,\gamma}(\bar u),0 \right)\right\rangle=0 \qquad \text{ for any } k, \gamma >1.$$

\medskip
\noindent
By Sobolev-Gagliardo-Nirenberg inequality, there exists $c:=c(p, N)>0$ such that

$$ \displaystyle \left( \into |u|^{p^*} \right)^{\frac{p}{p^*}} \leq c \into | \nabla u|^p   \qquad \text{ for any } u \in W_0^{1,p}(\Omega).$$

\medskip
\noindent
Furthermore, taking into account \eqref{contoconvessità}, \eqref{gradhkgamma}, \eqref{gradPhikgamma} and (\ref{disr}), we have 

\begin{align*}
	&\left(\into |h_{k,\gamma}(\bar u)|^{p^*}\right)^{\frac
		p{p^*}}\leq c \into |\nabla h_{k,\gamma}(\bar u)|^{p} = c \into |\nabla \bar
	u|^{p}|h'_{k,\gamma}(\bar u)|^{p} \nonumber
	\\
	&
	\leq c \into (\alpha^2 + |\nabla \bar u|^2)^{\frac{p-2}{2}}|\nabla \bar
	u|^{2} \ |h'_{k,\gamma}(\bar u)|^{p} + c \, \alpha^{p} \into |h'_{k,\gamma}(\bar u)|^{p} \nonumber\\
	&
\leq \frac{c}{\nu_1} \into \nabla \Psi_1 ( \nabla \bar{u} ) \cdot \nabla \bar{u} \ |h'_{k,\gamma}(\bar u)|^{p} + c \, \alpha^{p} \into |h'_{k,\gamma}(\bar u)|^{p}\nonumber\\
	&
= \frac{c}{\nu_1} \into \nabla \Psi_1 ( \nabla \bar{u} ) \cdot \nabla \Phi_{k,\gamma p,\gamma }(\bar u)   + c \, \alpha^{p} \into |h'_{k,\gamma}(\bar u)|^{p} \nonumber\\
	& = \frac{c}{\nu_1} \into H_s(\delta,x,\bar u,\bar v) \Phi_{k,\gamma p,\gamma }(\bar u)
	+ c \, \alpha^{p}\into |h'_{k,\gamma}(\bar u)|^{p}  \\
		& \leq \tilde{c}\into \left| H_s(\delta,x,\bar u,\bar v) \right| \left| \Phi_{k,\gamma p,\gamma }(\bar u) \right| 
	+\tilde{c} \into |h'_{k,\gamma}(\bar u)|^{p} .
\end{align*}

\medskip
\noindent
By $(\mathcal{H})$, we get
\begin{align}\label{integ1}
&	\left(\into |h_{k,\gamma}(\bar u)|^{p^*}\right)^{\frac	p{p^*}} \\
&	\leq \hat{c} \left(\into (|\bar u|^{p^*-1}+1)|\Phi_{k,\gamma p,\gamma }(\bar u)|
	+ \into |\bar v|^{q^*\frac{p^*-1}{p^*}}|\Phi_{k,\gamma p,\gamma }(\bar u)| 
	\right) \nonumber \\
&	+ \hat{c} \, \into |h'_{k,\gamma}(\bar u)|^{p}.
	\nonumber
\end{align}

\medskip
\noindent
By (\ref{hnew}) and H\"older's inequality with $\frac{p^*\gamma}{p(\gamma-1)}$ and $ \left( \frac{p^*\gamma}{p(\gamma-1)} \right)'$, we get

\begin{equation}\label{hne}
	\hat{c} \, \into |h'_{k,\gamma}(\bar u)|^{p}\leq \hat{c} \, \gamma^p \,\into |h_{k,\gamma}(\bar u)|^{p \frac{\gamma - 1}{\gamma}} \leq 
	c_0 \left(\into |h_{k,\gamma}(\bar u)|^{p^*}\right)^{\frac{p}{p^*}\frac{\gamma-1}{\gamma}}
\end{equation}

where $c_0>0$ depends on  $C_0, \left| \Omega \right|, p, N, \alpha, \nu_1$ and $\gamma$  but not on $k$ and $\delta$.

\smallskip
\noindent
For any $\sigma>1$, $r>1$ and $w$ in $\sobr$, we denote by
\[\Omega_{\sigma,w}=\{x\in \Omega\ : \ |w(x)| > \sigma\}.\]

\medskip
\noindent
From now on, when necessary, we redefine the positive constant $c_1$,  depending on $C_0,\left| \Omega \right|, p,q, N, \alpha, \nu_1$   and $\gamma$  but not on $k$ and $\delta$. Therefore, using \eqref{desgv}, \eqref{desgv2} together with H\"older inequality with conjugated exponents $\frac{p^* \gamma}{p \gamma + 1 - p}$ and $\left( \frac{p^* \gamma}{p \gamma + 1 - p} \right)'$, and with conjugated exponents $\frac{p^*}{p^*-p}$ and  $\frac{p^*}{p}$ respectively, we have

\begin{align}\label{integ2}
\displaystyle
& \into (|\bar u|^{p^*-1}+1)|\Phi_{k,\gamma p,\gamma }(\bar u)| \\
& = \into |\Phi_{k,\gamma p,\gamma }(\bar u)| + \int_{\Omega \setminus \Omega_{\sigma,\bar{u}}} |\bar u|^{p^*-1}|\Phi_{k,\gamma p,\gamma }(\bar u)| + \int_{\Omega_{\sigma,\bar{u}}} |\bar u|^{p^*-1}|\Phi_{k,\gamma p,\gamma }(\bar u)| \nonumber \\
& \leq (\sigma^{p^*-1}+1)\into|\Phi_{k,\gamma p,\gamma }(\bar u)| +\int_{\Omega_{\sigma,\bar u}} |\bar u|^{p^*-p} |\bar u|^{p-1} |\Phi_{k,\gamma p,\gamma }(\bar u)| \nonumber\\
&\leq 2\sigma^{p^*-1}\into|\Phi_{k,\gamma p,\gamma }(\bar u)|  + c_1\int_{\Omega_{\sigma,\bar u}} |\bar u|^{p^*-p} |h_{k,\gamma}(\bar u)|^p \nonumber \\
&\leq c_1\sigma^{p^*-1}\into |h_{k,\gamma}(\bar u)|^{\frac{p\gamma +1 -p}{\gamma}} +c_1\int_{\Omega_{\sigma,\bar u}} |\bar u|^{p^*-p} |h_{k,\gamma}(\bar u)|^p. \nonumber \\
&\leq c_1 \sigma^{p^*-1}\left(\into |h_{k,\gamma}(\bar u)|^{p^*}\right)^{\frac{p}{p^*}\frac{\gamma p+1-p}{\gamma p}} + c_1\|\bar u  \|^{p^*-p}_{L^{p^*}(\Omega_{\sigma,\bar u})}  \left(\into
	\left|h_{k,\gamma }(\bar u)\right|^{p^*}\right)^{\frac p {p^*}}. \nonumber 
\end{align}

\noindent
Now, we deal with the integral $\into |\bar v|^{q^*\frac{p^*-1}{p^*}}|\Phi_{k,\gamma p,\gamma }(\bar u)| $ in \eqref{integ1}.
By using \eqref{desgv}, \eqref{desgv2}, \eqref{desgv3}, the fact that $\Phi_{k,\gamma p,\gamma}(s)$ is non
decreasing for $s \geq 0$ and H\"older inequality as before, we obtain

\begin{align*} 
	&\into |\bar v|^{q^*\frac{p^*-1}{p^*}}|\Phi_{k,\gamma p,\gamma }(\bar u)|\\
	& = \int_{ \Omega \setminus \Omega_{\sigma,\bar v}}  |\bar v|^{q^*\frac{p^*-1}{p^*}}|\Phi_{k,\gamma p,\gamma }(\bar u)| + \int_{ \Omega_{\sigma,\bar v}} |\bar v|^{\frac{q^*}{p^*}(p^*-p)}|\bar v|^{\frac{q^*}{p^*}(p-1)}|\Phi_{k,\gamma p,\gamma }(\bar u)| \\
	&\leq  \sigma^{q^*\frac{p^*-1}{p^*}}\into |\Phi_{k,\gamma p,\gamma }(\bar u)| 
	+\int_{\Omega_{\sigma,\bar v}\cap \{|\bar v|^\frac{q^*}{p^*}\leq |\bar u|\}} |\bar v|^{\frac{q^*}{p^*}(p^*-p)} |\bar u|^{p-1} |\Phi_{k,\gamma p,\gamma }(\bar u)| 
	\\
	&+\int_{\Omega_{\sigma,\bar v}\cap \{|\bar v|^\frac{q^*}{p^*}\geq |\bar u|\}} |\bar v|^{\frac{q^*}{p^*}(p^*-p)} \bigl(|\bar v|^{\frac{q^*}{p^*}}\bigr)^{p-1} |\Phi_{k,\gamma p,\gamma }(|\bar v|^{\frac{q^*}{p^*}})| 
	\\
	&\leq c_1\sigma^{q^*\frac{p^*-1}{p^*}}\into |h_{k,\gamma}(\bar u)|^{\frac{p\gamma +1 -p}{\gamma}}
	+c_1 \int_{\Omega_{\sigma,\bar v}\cap \{|\bar v|^\frac{q^*}{p^*}\leq |\bar u|\}} |\bar v|^{\frac{q^*}{p^*}(p^*-p)}|h_{k,\gamma}(\bar u)|^p
	\\
	&+c_1 \int_{\Omega_{\sigma,\bar v}\cap \{|\bar v|^\frac{q^*}{p^*}\geq |\bar u|\}} |\bar v|^{\frac{q^*}{p^*}(p^*-p)}
	|h_{k,\gamma}\bigl(|\bar v|^{\frac{q^*}{p^*}}\bigr)|^p
	\\
	&\leq c_1\sigma^{q^*\frac{p^*-1}{p^*}} 
	\Bigl(\into |h_{k,\gamma
	}(\bar u)|^{p^*}\Bigr)^{\frac{p}{p^*}\frac{\gamma p+1-p}{\gamma p}}\\
	&+c_1 \|\bar v \|^{\frac{q^*}{p^*}(p^*-p)}_{L^{q^*}(\Omega_{\sigma,\bar v})} 
	\left(	
	\left(\into
	\left|h_{k,\gamma }(\bar u)\right|^{p^*}\right)^{\frac p {p^*}}
	+\left(\int_{\Omega}
	\bigg| h_{k^{\frac{p^*}{q^*}},\gamma }\left(\bar v\right)\bigg|^{q^*}\right)^{\frac p {p^*}}
	\right).
	\nonumber
\end{align*}

\medskip
\noindent
Combining with  \eqref{integ1}, \eqref{hne} and \eqref{integ2}, we get

\begin{align*}
	&\left(\into |h_{k,\gamma}(\bar u)|^{p^*}\right)^{\frac
		p{p^*}} \\
	&\leq c_1
	\sigma^{p^*-1}\left(\into |h_{k,\gamma
	}(\bar u)|^{p^*}\right)^{\frac{p}{p^*}\frac{\gamma p+1-p}{\gamma p}}
	+ c_1\|\bar u  \|^{p^*-p}_{L^{p^*}(\Omega_{\sigma,\bar u})}  \left(\into
	\left|h_{k,\gamma }(\bar u)\right|^{p^*}\right)^{\frac p {p^*}}
	\\
	&+ c_1\sigma^{\frac{q^*}{p^*}(p^*-1)} 
	\left( \into |h_{k,\gamma
	}(\bar u)|^{p^*}\right)^{\frac{p}{p^*}\frac{\gamma p+1-p}{\gamma p}}\\
	&+c_1 \|\bar v \|^{\frac{q^*}{p^*}(p^*-p)}_{L^{q^*}(\Omega_{\sigma,\bar v})} 
	\left(	
	\left(\into
	\left|h_{k,\gamma }(\bar u)\right|^{p^*}\right)^{\frac p {p^*}}
	+\left(\into
	\bigg| h_{k^{\frac{p^*}{q^*}},\gamma }\left(\bar v\right) \bigg|^{q^*}\right)^{\frac p {p^*}}
	\right)\\
	&+c_1 \left(\into |h_{k,\gamma}(\bar u)|^{p^*}\right)^{\frac{p}{p^*}\frac{\gamma-1}{\gamma}},
\end{align*}

\medskip
\noindent
where the constant $c_1$ depends on $C_0,\left| \Omega \right|, p,q, N, \alpha, \nu_1$  and $\gamma$  but not on $k$ and $\delta$.

\bigskip
\noindent
By using Lemma \ref{V} with $\varepsilon=(4c_1)^{\frac{p^*}{p-p^*}}$ and $u_0 \in W_0^{1,p}(\Omega)$, and  with $\varepsilon=(4 c_1)^{\frac{p^*}{p-p^*}}$ and $v_0 \in W_0^{1,q}(\Omega)$ respectively, we infer that there exists  $\sigma_1>1$ and $R_1>0$ such that, for any $\sigma\geq \sigma_1$ and for any $(\bar u,\bar v)\in D_{R_1, \delta}(u_0,v_0),$

$$ \displaystyle   c_1\|\bar u  \|^{p^*-p}_{L^{p^*}(\Omega_{\sigma,\bar u})} + c_1 \|\bar v \|^{\frac{q^*}{p^*}(p^*-p)}_{L^{q^*}(\Omega_{\sigma,\bar v})} \leq \frac{1}{2}.$$

\medskip
\noindent
Hence, we infer that 

\begin{align}\label{k0}
\displaystyle
        \frac{1}{2} \left(\into |h_{k,\gamma}(\bar u)|^{p^*}\right)^{\frac{p}{p^*}} 
& \leq  c_1 \,  \Bigl(\sigma^{p^*-1}+\sigma^{\frac{q^*}{p^*}(p^*-1)} \Bigr) \left(\into |h_{k,\gamma}(\bar u)|^{p^*}\right)^{\frac{p}{p^*}\frac{\gamma p+1-p}{\gamma p}} \\
&	+  c_1  \,  \| \bar v \|^{\frac{q^*}{p^*}(p^*-p)}_{L^{q^*}(\Omega_{\sigma,\bar v})} \left(\into
	\bigg| h_{k^{\frac{p^*}{q^*}},\gamma }\left(\bar v\right) \bigg|^{q^*}\right)^{\frac p {p^*}}\nonumber\\
&	+  c_1  \,  \left(\into |h_{k,\gamma}(\bar u)|^{p^*}\right)^{\frac{p}{p^*}\frac{\gamma-1}{\gamma}} . \nonumber
\end{align}

\medskip
\noindent
If $\eta \in (0,1)$, by using Young inequality with conjugated exponents $\frac{1}{\eta}$ and $ \frac{1}{1-\eta}$, we obtain that

\begin{equation}\label{inyou}
\displaystyle	ax^{\eta} \leq \left( \frac{x}{8} \right)^{\eta} 8a \leq \frac x 8 + 
	(8a)^{1/(1-\eta)} \qquad \forall \, a,x \geq 0.
\end{equation}

\bigskip
\noindent
Therefore,  since  $ \frac{\gamma p+1-p}{\gamma p}<1$ and $ \frac{\gamma-1}{\gamma}<1$, we infer

\begin{align*}
	c_1 \Bigl(\sigma^{p^*-1}+\sigma^{\frac{q^*}{p^*}(p^*-1)}\Bigr)
&	\left(\into |h_{k,\gamma
	}(\bar u)|^{p^*}\right)^{\frac{p}{p^*}\frac{\gamma p+1-p}{\gamma p}}
	\hspace{-4mm}\\
& \leq \ \frac 1 8 \left(\into |h_{k,\gamma
	}(\bar u)|^{p^*}\right)^{\frac{p}{p^*}}+c_1\, \Bigl(\sigma^{p^*-1}+\sigma^{\frac{q^*}{p^*}(p^*-1)}\Bigr)^{\frac{\gamma p}{p-1}}
\end{align*}

\noindent
and 

\[c_1 \left(\into |h_{k,\gamma}(\bar u)|^{p^*}\right)^{\frac{p}{p^*}\frac{\gamma-1}{\gamma}}
\leq \frac 1 8 \left(\into |h_{k,\gamma
}(\bar u)|^{p^*}\right)^{\frac{p}{p^*}}+c_1,
\]

\bigskip
\noindent
so that \eqref{k0} becomes

\begin{align*} 
	&\frac 1 4\left(\into |h_{k,\gamma}(\bar u)|^{p^*}\right)^{\frac
		p{p^*}}\\
	&\leq 
	c_1 \left[ \Bigl(\sigma^{p^*-1}+\sigma^{\frac{q^*}{p^*}(p^*-1)}\Bigr)^{\frac{\gamma p}{p-1}} + 1 \right]
	+  c_1 \, \|\bar v \|^{\frac{q^*}{p^*}(p^*-p)}_{L^{q^*}(\Omega_{\sigma,\bar v})} 
	\left(\into
	\bigg|h_{k^{\frac{p^*}{q^*}},\gamma }\left(\bar v\right)\bigg|^{q^*}\right)^{\frac p {p^*}}\hspace{-3mm}.
\end{align*}

\bigskip
\noindent
Thus,  we have shown that for any 
$k,\gamma >1$ there exists a positive constant $c_1$ depending on $C_0,\left| \Omega \right|, p,q, N, \alpha, \nu_1$  and $\gamma$  but not on $k$ and $\delta$, and there are  $\sigma_1>1$ and a suitably small $R_1>0$, depending on $c_1$ and $(u_0,v_0)$,  such that for any $(\bar u, \bar v) \in D_{R_1,  \delta}(u_0,v_0)$ and any $\sigma\geq \sigma_1$ we have

\begin{align}\label{integ4}
	\ &\into |h_{k,\gamma}(\bar u)|^{p^*}\\
	&\leq 
	c_1\left[\Bigl(\sigma^{p^*-1}+\sigma^{\frac{q^*}{p^*}(p^*-1)}\Bigr)^{\frac{\gamma p^*}{p-1}}
	 + 1 \right] +  c_1\, \|\bar v \|^{\frac{q^*}{p}(p^*-p)}_{L^{q^*}(\Omega_{\sigma,\bar v})} 
	\into
	\bigg|h_{k^{\frac{p^*}{q^*}},\gamma }\left(\bar v\right)\bigg|^{q^*}. \nonumber
\end{align}

\bigskip
\noindent
Reasoning in a similar way, exploiting that 

$$\left\langle I'_{\delta,\Psi_1,\Psi_2}(\bar z), \left(0,\Phi_{\tilde k,\gamma q,\gamma}(\bar v) \right)\right\rangle=0 \qquad \text{ for any } \tilde k, \gamma>1,$$

\medskip
\noindent
we infer that for any $\tilde{k},\gamma >1$ there exists a positive constant $c_2$ depending on $C_0,\left| \Omega \right|, p,q, N, \beta, \nu_2$  and $\gamma$  but not on $\tilde{k}$ and $\delta$, and there are  $\sigma_2>1$ and a suitably small $R_2>0$, depending on $c_2$ and $(u_0,v_0)$,  such that for any $(\bar u, \bar v) \in D_{R_2,  \delta}(u_0,v_0)$ and any $\sigma\geq \sigma_2$ we have

\begin{align}\label{integ5}
	& \into |h_{\tilde k,\gamma}(\bar v)|^{q^*}\\
	&\leq c_2 \left[ \Bigl(\sigma^{q^*-1}+\sigma^{\frac{p^*}{q^*}(q^*-1)} \Bigr)^{\frac{\gamma q^*}{q-1}}
	+1 \right] +  c_2 \, \|\bar u \|^{\frac{p^*}{q}(q^*-q)}_{L^{p^*}(\Omega_{\sigma,\bar u})} 
	\into
	\bigg|h_{\tilde k^{\frac{q^*}{p^*}},\gamma }\left(\bar u\right)\bigg|^{p^*}.
	\nonumber
\end{align}

\bigskip
\noindent
Let us consider a suitable $\bar{\sigma}>1$ and a suitably small $\bar{R}>0$ such that both \eqref{integ4} and \eqref{integ5} hold for any $\sigma \geq \bar{\sigma}$ and for any $0<R \leq \bar{R}$.\\
Setting $\tilde{k}=k^{\frac{p^*}{q^*}}$ in (\ref{integ5}) and substituting in (\ref{integ4})  we obtain

\begin{align*}
\displaystyle
    \into |h_{k,\gamma}(\bar u)|^{p^*} 
&	\leq c_1 \left[ \Bigl(\sigma^{p^*-1}+\sigma^{\frac{q^*}{p^*}(p^*-1)} \Bigr)^{\frac{\gamma p^*}{p-1}} + 1 \right]\\
&	+  c_1 \, c_2 \, \|\bar v \|^{\frac{q^*}{p}(p^*-p)}_{L^{q^*}(\Omega_{\sigma,\bar v})} 
	 \left[ \Bigl(\sigma^{q^*-1}+\sigma^{\frac{p^*}{q^*}(q^*-1)} \Bigr)^{\frac{\gamma q^*}{q-1}}
	+ 1 \right] \nonumber  \\
&  + c_1 \, c_2 \, \|\bar v \|^{\frac{q^*}{p}(p^*-p)}_{L^{q^*}(\Omega_{\sigma,\bar v})}  \, \|\bar u \|^{\frac{p^*}{q}(q^*-q)}_{L^{p^*}(\Omega_{\sigma,\bar u})} \into|h_{k,\gamma }\left(\bar u\right)|^{p^*}.
\end{align*}

\medskip
\noindent
By using again Lemma \ref{V} with $\varepsilon=2^{\frac{q}{q-q^*}}$ and $u_0 \in W_0^{1,p}(\Omega)$, and  with $\varepsilon=(c_1c_2)^{\frac{p}{p-p^*}}$ and $v_0 \in W_0^{1,q}(\Omega)$ respectively, we infer that there exists  $\tilde{\sigma} \geq \bar{\sigma}$ and $ \tilde{R}  \in (0,\bar{R}]$ such that, for any $\sigma\geq \tilde{\sigma}$ and for any $(\bar u,\bar v)\in D_{\tilde{R}, \delta}(u_0,v_0),$

$$ \displaystyle  c_1 c_2\, \|\bar v \|^{\frac{q^*}{p}(p^*-p)}_{L^{q^*}(\Omega_{\sigma,\bar v})} \leq 1 \quad \text{and} \quad \|\bar u \|^{\frac{p^*}{q}(q^*-q)}_{L^{p^*}(\Omega_{\sigma,\bar u})} \leq \frac{1}{2}.$$

\medskip
\noindent
Hence, we finally get that for any $(\bar u, \bar v) \in D_{R, \delta}(u_0,v_0)$

\[\into |h_{k,\gamma}(\bar u)|^{p^*} \leq c_3  \qquad \text{ for any } k,\gamma>1, \]

\noindent
and similarly

\[\into |h_{k,\gamma}(\bar v)|^{q^*} \leq c_3 \qquad \text{ for any } k,\gamma>1, \]

\medskip
\noindent
where the positive constant $c_3$ depends on $C_0,\left| \Omega \right|, p,q, N, \alpha, \nu_1, \beta, \nu_2$  and $\gamma$  but not on $k$ and $\delta$, and in addition $R>0$ is a suitable radius.

\medskip
\noindent
Thus we can apply Fatou Lemma and, passing to limit for $k \to \infty$, we get 

\begin{equation*}
	\into |\bar u|^{\gamma  p^*} \leq C \quad \text{and} \quad \into |\bar v|^{\gamma  q^*} \leq  C 
	\qquad \qquad \forall (\bar u, \bar v) \in D_{R, \delta}(u_0,v_0),
\end{equation*}

\noindent
where the positive constant $C$ still depends on $C_0,\left| \Omega \right|, p,q, N, \alpha, \nu_1, \beta, \nu_2$  and $\gamma$ but not on $\delta$.
\end{proof}

\medskip\noindent
{\mbox {\it Proof of Theorem~\ref{key}.~}} From now on, we will omit $dx$ and we will redefine the positive constant $C$ when necessary.\\
Fixing $m>N$, by Proposition \ref{integrability} and assumption  $(\mathcal{H})$ we infer that there are $C, R >0$ depending on $C_0,\left| \Omega \right|, p,q, N, \alpha, \nu_1, \beta$,$\nu_2$ and $m$  but not on $\delta$, such that  $H_s(\delta,x,\bar u,\bar v) \in L^m(\Omega)$ and $\, H_t(\delta,x,\bar u,\bar v)\in L^m(\Omega)$ for any $(\bar u, \bar v) \in D_{R, \delta}(u_0,v_0)$, and
\begin{equation}\label{dis1}
\displaystyle \|H_s(\delta,x,\bar u,\bar v)\|_m^{p'}, \quad  \|H_t(\delta,x,\bar u,\bar v)\|_m^{q'} \leq C,
\end{equation}
\noindent
where $p'$ and $q'$  are the conjugated exponents of $p$ and $ q$ respectively.

\medskip
\noindent 
Fixed $(\bar u, \bar v) \in D_{R, \delta}(u_0,v_0)$, for any $k \geq 1$ we denote by  
$$A^{\bar{u}}_k= \left\{ x\in \Omega\ :\ | \bar{u} (x)| > k\right\},$$
and we define $G_k: \mathbb{R} \to \mathbb{R}$  as

$$ \displaystyle
G_k(r):=
\begin{cases}
0 & \text{ if } |r| \leq k, \bigskip \\
 \text{sign}(r) \Bigl[ \, |r| - k \, \Bigr] & \text{ if } |r| > k. \\
\end{cases}
$$

\noindent
Since $G_k(\bar{u}) \in W_0^{1,p}(\Omega)$ and $\langle I'_{\delta,\Psi_1,\Psi_2}(\bar{u},\bar{v}),\left(G_k(\bar u),0\right)\rangle=0,$  by using \eqref{contoconvessità}, \eqref{disr},   and  denoting by $\bar{f}_s:=H_s(\delta,x, \bar{u}, \bar{v}) $, we have

\begin{align}\label{calc1}
\displaystyle
&  \int_{A^{\bar{u}}_k} \left| \nabla \bar{u} \right|^p = \int_{A^{\bar{u}}_k} \left| \nabla G_k( \bar{u}) \right|^p =  \int_{A^{\bar{u}}_k} \left| \nabla \bar{u} \right|^p G'_k(\bar{u})^p = \int_{A^{\bar{u}}_k} \left| \nabla \bar{u} \right|^p G'_k(\bar{u}) \\
& \leq \int_{A^{\bar{u}}_k} \left( \alpha^2 + \left| \nabla \bar{u} \right|^2  \right)^{\frac{p-2}{2}} \left| \nabla \bar{u}  \right|^2 G'_k(\bar{u}) + \alpha^{p} \int_{A^{\bar{u}}_k} G'_k(\bar{u})  \nonumber \\
& \leq \into \left( \alpha^2 + \left| \nabla \bar{u} \right|^2  \right)^{\frac{p-2}{2}} \left| \nabla \bar{u}  \right|^2 G'_k(\bar{u})   + \alpha^{p} \left| A^{\bar{u}}_k \right|  \nonumber \\
& \leq \frac{1}{\nu_1} \into \nabla \Psi_1 ( \nabla \bar{u} ) \cdot \nabla \bar{u} \, G'_k(\bar{u})  + \alpha^{p} \left| A^{\bar{u}}_k \right| \nonumber \\
& = \frac{1}{\nu_1} \into \nabla \Psi_1 ( \nabla \bar{u} ) \cdot \nabla G_k(\bar{u})  + \alpha^{p} \left| A^{\bar{u}}_k \right| \nonumber \\
& = \frac{1}{\nu_1} \into \bar{f}_s \, G_k( \bar{u}) + \alpha^{p} \left| A^{\bar{u}}_k \right| \nonumber \\
& = \frac{1}{\nu_1} \int_{A^{\bar{u}}_k} \bar{f}_s \, G_k( \bar{u}) + \alpha^{p} \left| A^{\bar{u}}_k \right| \nonumber \\
& \leq  \frac{1}{\nu_1} \int_{A^{\bar{u}}_k} \left| \bar{f}_s \right| \left| G_k(\bar{u})  \right|  + \alpha^{p} \left| A^{\bar{u}}_k \right|. \nonumber  
\end{align}

\medskip
\noindent
By Sobolev-Gagliardo-Nirenberg inequality,  H\"older inequality with conjugated exponents  $({p^*})'$ and  $p^*$, and Young inequality with conjugated exponents $p'$ and $p$, we have 

\begin{align*}
\displaystyle 
\int_{A^{\bar{u}}_k} \left| \bar{f}_s \right| \left| G_k(\bar{u})  \right|  
& \leq \left(  \int_{A^{\bar{u}}_k} \left| \bar{f}_s \right|^{{p^*}'}  \right)^{\frac{1}{{p^*}'}} \left( \int_{A^{\bar{u}}_k} \left| G_k(\bar{u})  \right|^{p^*}  \right)^{\frac{1}{p^*}}   \\
& = \left(  \int_{A^{\bar{u}}_k} \left| \bar{f}_s \right|^{{p^*}'}  \right)^{\frac{1}{{p^*}'}} \left( \into \left| G_k(\bar{u})  \right|^{p^*}  \right)^{\frac{1}{p^*}}  \nonumber \\
& \leq  C \left(  \int_{A^{\bar{u}}_k} \left| \bar{f}_s \right|^{{p^*}'}  \right)^{\frac{1}{{p^*}'}} \left( \into \left|  \nabla G_k(\bar{u})  \right|^{p}  \right)^{\frac{1}{p}} \nonumber \\
& =  C \left( \int_{A^{\bar{u}}_k}\left| \bar{f}_s \right|^{{p^*}'} \right)^{\frac{1}{{p^*}'}} \left( \int_{A^{\bar{u}}_k} \left|  \nabla \bar{u}  \right|^{p}   \right)^{\frac{1}{p}} \nonumber  \\
& \leq \frac{1}{p'} C^{p'} \left( \int_{A^{\bar{u}}_k}\left| \bar{f}_s \right|^{{p^*}'}  \right)^{\frac{p'}{{p^*}'}} + \frac{1}{p}  \int_{A^{\bar{u}}_k} \left|  \nabla \bar{u}  \right|^{p} .  \nonumber
\end{align*}

\medskip
\noindent
Therefore, we have 

\begin{equation}\label{nuovoconto}
\displaystyle \frac{1}{\nu_1} \int_{A^{\bar{u}}_k} \left| \bar{f}_s \right| \left| G_k(\bar{u})  \right| \, \leq \frac{1}{ \nu_1 \, p'} C^{p'} \left( \int_{A^{\bar{u}}_k}\left| \bar{f}_s \right|^{{p^*}'}  \right)^{\frac{p'}{{p^*}'}} + \frac{1}{ \nu_1 \, p}  \int_{A^{\bar{u}}_k} \left|  \nabla \bar{u}  \right|^{p} . 
\end{equation}

\medskip
\noindent
Taking into account \eqref{calc1} and \eqref{nuovoconto}, we have proved

\begin{equation}\label{calc3}
\displaystyle \left( 1 - \frac{1}{\nu_1 \, p }  \right)\int_{A^{\bar{u}}_k} \left| \nabla \bar{u} \right|^p   \leq \frac{C^{p'}}{\nu_1 \, p'} \left(  \into \left| \bar{f}_s \right|^{{p^*}'} \right)^{\frac{p'}{{p^*}'}} +  \alpha^{p} \left| A^{\bar{u}}_k \right|,
\end{equation}

\medskip
\noindent
by which, considering that $\nu_1 > \frac{1}{p}$, we get

\begin{equation}\label{calcnu1}
\displaystyle\int_{A^{\bar{u}}_k} \left| \nabla \bar{u} \right|^p   \leq \frac{p}{p'(\nu_1 \, p - 1)}C^{p'}\left(  \into \left| \bar{f}_s \right|^{{p^*}'}  \right)^{\frac{p'}{{p^*}'}} +  \frac{\nu_1 \, p}{\nu_1 \, p -1}\alpha^{p} \left| A^{\bar{u}}_k \right|,
\end{equation}

\medskip
\noindent
Since $m>N$, we deduce that 

$$ \displaystyle m > {p^*}'=\frac{Np}{Np-N+p} \quad \text{and} \quad p' \left( \frac{1}{{p^*}'} - \frac{1}{m}  \right) >1.$$

%
%
%
%
%
%
%

\medskip
\noindent
Now, by  (\ref{dis1}) and  H\"older inequality with conjugated exponents $\displaystyle \frac{m}{{p^*}'}$ and $\displaystyle \frac{m}{{m-p^*}'}$, we have

\begin{align}\label{calc4}
\displaystyle
\left( \int_{A^{\bar{u}}_k} \left| \bar{f}_s \right|^{{p^*}'}  \right)^{\frac{p'}{{p^*}'}} 
& \leq  
  \left( \int_{A^{\bar{u}}_k} |\bar{f}_s|^m \right)^\frac{p'}{m} \, |A^{\bar{u}}_k|^{p' \left( \frac{1}{{p^*}'} - \frac{1}{m} \right)}.   \\
& =  \|H_s(\delta,x,\bar u,\bar v)\|_m^{p'} \, |A^{\bar{u}}_k|^{p' \left( \frac{1}{{p^*}'} - \frac{1}{m} \right)} \nonumber \bigskip\\
& \leq  C \, |A^{\bar{u}}_k|^{p' \left( \frac{1}{{p^*}'} - \frac{1}{m} \right)} .   \nonumber
\end{align}

\medskip
\noindent
Considering \eqref{calcnu1} and  \eqref{calc4}, we have proved that there exists $C>0$ such that

\begin{equation}\label{calc5}
\displaystyle \int_{A^{\bar{u}}_k} \left| \nabla \bar{u} \right|^p   \leq  C \left( \, |A^{\bar{u}}_k|^{p' \left( \frac{1}{{p^*}'} - \frac{1}{m} \right)} +   \,  \left| A^{\bar{u}}_k \right| \right).
\end{equation}

\medskip
\noindent
Now, by using  Lemma \ref{V2023}, there exists $k_0 \geq 1$ such that for any $(\bar u, \bar v) \in D_{R, \delta}(u_0,v_0)$ and for any $k$ greater then $k_0$, we get

$$ \displaystyle \left| A^{\bar{u}}_k  \right| < 1. $$

%
%
	
%
%
%
%
%

%
%
%

\medskip
\noindent
By \eqref{calc5} and $ \displaystyle p' \left( \frac{1}{{p^*}'} - \frac{1}{m}  \right) >1,$  for any $k \geq k_0$ we have

\begin{equation}\label{calc6}
\displaystyle \int_{A^{\bar{u}}_k} \left| \nabla \bar{u} \right|^p   \leq 2 \bar{C}  \,  \left| A^{\bar{u}}_k \right|.
\end{equation}

\medskip
\noindent
Therefore, by applying Theorem \ref{Mammoliti} with $\varepsilon= 1 - \frac{p}{p^*}$ and $\theta=0$, we deduce that $\bar u\in L^\infty(\Omega)$. Similarly, we deduce that  $\bar v\in L^\infty(\Omega)$.\\
Finally, we have proved that there exists a suitable small $R>0$ and a constant $C>0$ depending on $C_0,\left| \Omega \right|, p,q, N, \alpha, \nu_1, \beta$ and $\nu_2$  but not on $\delta$, such that for any $(\bar u, \bar v) \in D_{R,\delta}(u_0,v_0)$,
$\bar u$ and $\bar v$ are in $L^\infty(\Omega)$ and
\[\|\bar u\|_{\infty}, \ \|\bar v\|_{\infty}\leq C.
\]

\qed

\begin{remark}
We point out that if $ p \geq 2$ the constant $C$ does not depend on $\alpha$. In fact, in the proof of theorem \ref{key}  we can use instead of \eqref{disr} the inequality 

$$ \displaystyle 	s^p \leq \left(\alpha^2 +s^2\right)^{\frac{p-2}{2}}s^2, $$

\medskip
\noindent 
that holds for any $p \geq 2$ and for any $s \geq 0$.\\
Similarly, if $q \geq 2$  the constant $C$ does not depend on $\beta$.
\end{remark}

\medskip
\noindent
{\bf Author contributions}\\
The authors have accepted responsibility for the entire content of this manuscript and
approved its submission.

\medskip
\noindent
{\bf Competing interests} \\
The authors state no conflict of interest.

\medskip
\noindent
{\bf Acknowledgments}\\
The authors thank INdAM-GNAMPA. The second and third authors are supported by PRIN PNRR P2022YFAJH {\sl “Linear and Nonlinear PDEs: New directions and applications”}. 
The first and second authors thank PNRR MUR project CN00000013 HUB - National
Centre for HPC, Big Data and Quantum Computing (CUP H93C22000450007). The first author is also supported by INdAM-GNAMPA project { \sl “Critical and limiting phenomena in nonlinear elliptic systems”} (CUP E5324001950001).  The third author is also supported by the Italian Ministry of University and Research under the Programme “Department of Excellence” Legge 232/2016 (Grant No. CUP - D93C23000100001).

\medskip
\noindent
{\bf Data availability}\\
Not applicable.

\end{document}